\documentclass[letterpaper, 10 pt, conference]{ieeeconf}  

\pdfoutput=1
\usepackage{amssymb}
\usepackage{amsmath}
\usepackage{algorithm}
\usepackage{algorithmic}

\newtheorem{remark}{Remark}
\newtheorem{theorem}{Theorem}

 \usepackage{graphicx}
 \usepackage {subcaption}
  \usepackage{cite} 
  \makeatletter
  \let\NAT@parse\undefined
  \makeatother
  \usepackage{hyperref}
  \hypersetup{hypertex=true,
  	colorlinks=true,
  	linkcolor=blue,
  	anchorcolor=blue,
  	citecolor=blue}
  
\IEEEoverridecommandlockouts                              

\title{\LARGE \bf
Optimization Method Based On Optimal Control 
}

\author{Yeming Xu, Ziyuan Guo, Hongxia Wang, and Huanshui Zhang$^*$, Senior Member, IEEE
\thanks{This work was supported by the Foundation for Innovative Research Groups of the National Natural Science Foundation of China (61821004), Major Basic Research of Natural Science Foundation of Shandong Province (ZR2021ZD14),  High-level Talent Team Project of Qingdao West Coast New Area (RCTD-JC-2019-05), Key Research and Development Program of Shandong Province (2020CXGC01208), and Original Exploratory Program Project of National Natural Science Foundation of China (62250056) (Corresponding author: Huanshui Zhang).
}
\thanks{Yeming Xu, Ziyuan Guo, Hongxia Wang, and Huanshui Zhang are with the College
	of Electrical Engineering and Automation, Shandong University of Science
	and Technology, Qingdao, 266590, China (e-mail: ymxu2022@163.com; skdgzy@sdust.edu.cn;
	whx1123@163.com; hszhang@sdu.edu.cn).
	      }
}

\begin{document}

\maketitle
\thispagestyle{empty}
\pagestyle{empty}

\begin{abstract}
In this paper, we focus on a method based on optimal control to address the optimization problem. The objective is to find the optimal solution that minimizes the objective function.  We transform the optimization problem into optimal control by designing an appropriate cost function. Using Pontryagin's Maximum Principle and the associated forward-backward difference equations (FBDEs), we derive the iterative update gain for the optimization. The steady system state can be considered as the solution to the optimization problem. Finally, we discuss the compelling characteristics of our method and further demonstrate its high precision, low oscillation, and applicability for finding different local minima of non-convex functions through several simulation examples.
\end{abstract}
\section{INTRODUCTION}
Optimization problems, which involve the search for the minimum of a specified objective function, play a crucial role in various fields, including engineering, economics, machine learning, etc. \cite{luenberger1984linear}. Optimization methods are the basis for solving various optimization problems such as system identification and optimal control. Therefore, optimization problems have attracted extensive attention in various fields over the past few centuries, leading to significant advancements as follows:
%

Gradient descent stands as the oldest, most predominant, and most effective first-order method for tackling optimization problems. Its simplicity captured widespread attention upon its inception.
Typical gradient descent techniques include exact and inexact line search \cite{boyd2004convex,nocedal1999numerical}, and more. With the development of artificial intelligence technology, gradient descent has gained new vitality. This revival includes the emergence of various techniques tailored to different optimization needs.
These encompass batch gradient descent (BGD) \cite{ruder2016overview}, which operates on the entire training set, and mini-batch gradient descent (MBGD) \cite{huo2017asynchronous}, which processes subsets of training data.
 Notably, stochastic gradient descent (SGD) \cite{bottou2010large}, using a data in training set in each update, has also made a significant impact. Furthermore, the development of optimization methods has introduced enhancements to the traditional gradient descent approach such as Momentum gradient descent \cite{qian1999momentum}, Nesterov Momentum gradient descent \cite{o2015adaptive}, AdaGrad gradient descent \cite{duchi2011adaptive}, Adam gradient descent \cite{kingma2014adam}, to name a few. Nevertheless, gradient descent still faces issues such as slow convergence near extremal points, susceptibility to oscillations, and difficulty in finding optimal points.

 Newton's method emerges as the most basic and effective second-order method for solving optimization problems. Owing to its exceptional precision and fast convergence, it is very favorable, leading to various improved versions of Newton's method, including modified Newton's method \cite{fletcher1977modified}, damped Newton's method \cite{sano2019damped}, and quasi-Newton methods \cite{broyden1967quasi,gill1972quasi}. Notable algorithms in this category are DFP \cite{broyden1970convergence}, BFGS \cite{broyden1970convergence2}, and L-BFGS \cite{liu1989limited}. In recent years, Newton's method and its improved versions have also been widely applied in training neural networks for machine learning and solving large-scale logistic regression problems \cite{setiono1995use,goldfarb2020practical,lin2007trust}. However, Newton's method and its variants may encounter the following challenges: (a) the need for an initial value close enough to the extremal point, or it may diverge; (b) strict requirements for the objective function, necessitating second-order partial derivatives; (c) for multivariable optimization, the calculation of the inverse matrix of the Hessian matrix is computationally burdensome.

There are several other algorithms like Conjugate Gradient \cite{luenberger1984linear} and Evolutionary Algorithms \cite{back1996evolutionary}, used for addressing optimization problems. We won't list them exhaustively. Unfortunately, these algorithms, while valuable in many situations, may still present some challenges such as slow convergence, oscillations during convergence, susceptibility to divergence, applicability only to functions with specific structures or under certain algorithm parameter settings, and inefficiency in handling non-convex optimization.

Unlike gradient descent and Newton's method, we propose a novel optimization idea for the optimization problem by addressing a new optimal control problem. It aims to design an optimal controller to regulate a first-order difference equation such that the cost function, closely related to the objective function, is minimized.
 The optimal trajectory of the system to rapidly approach the local minimum point of the original optimization objective. This method offers relatively flexible initial value selection, fast convergence speed, effective avoidance of oscillations observed in gradient methods, and does not require the computation of second-order partial derivatives of the original optimization objective function. Particularly, for some nonconvex functions, different local minimum points can be obtained by adjusting the input weight matrix of the optimal control problem, provided that the initial value is chosen properly. It's worth noting that the selection of the input weight matrix does not lead to divergence and oscillations compared to conventional algorithms.

We use standard notation: $\mathbb{R}^n$ is the set of n-dimensional real vectors; $\mathbb{S}^n_{++}$ is
the set of positive definite symmetric matrices; $I_n$ is the n-dimensional identity matrix; $A \succ B (A\prec B)$ means that the matrix $A-B$ is positive (negative) definite, we said A is larger (smaller) than B; $\nabla f({x})$ and $\nabla^2 f({x})$ denote the gradient and the hessian matrix of $f({x})$.

The remainder of this paper is organized as follows. In Section II, we formulate the optimization problem as an optimal control problem. In Section III, we approach it by Pontryagin's Maximum Principle and summarize the characteristics of our proposed method. In Section IV, we conduct simulations to validate our results in both convex and un-convex settings. Concluding remarks of Section V complete the paper.

\section{PROBLEM FORMULATION}
Consider the optimization problem
\begin{equation}
		\mathop {\rm minimize}\limits_{x \in {\mathbb{R}^n}} f(x), \label{1}
\end{equation}
where $f:{\mathbb{R}^n} \to \mathbb{R}$ is a nonlinear function. The objective function $f(x)$ is assumed to be twice continuously differentiable on ${\mathbb{R}^n}$.

Numerous algorithms have been developed to address the minimization problem, many of which are grounded in the principle of gradient descent \cite{ruder2016overview} i.e.,
\begin{equation}
{x_{k + 1}} = {x_k} + \eta\nabla f({x_k}),\label{2}
\end{equation}where $\eta$ is step size.

On one hand, it has been demonstrated in convex optimization that the global optimal solution can be obtained using (\ref{2}) \cite{boyd2004convex}. However, gradient descent is highly sensitive to the choice of step size. The smaller step size ensures convergence during the iterative process, but this comes at the expense of sacrificing convergence speed. Conversely, the larger step size may easily lead to oscillations and divergence during the iterative process. One of the most fundamental methods for determining the step size is line search criterion \cite{nocedal1999numerical}.

On the other hand, in the case of non-convex optimization, the search for the global optimal solution remains challenging. Different initial points and algorithms may lead to different local optimal solutions or oscillate and diverge in the iterative process. Moreover, there are often fewer guarantees to prove the existence and properties of an optimal solution, making algorithm design and analysis more complex. Due to the inherent challenges of effectively solving non-convex optimization problems, the primary methods currently employed to address such problems include: (a) Find problems with implicit convexity, or solve them by convex reconstruction. (b) The target from finding global solution changes for a stationary point or local extremum points. (c) Consider a class of non-convex problems that can provide global performance guarantees, such as satisfying the Polyak-Łojasiewicz condition \cite{danilova2022recent,polyak1963gradient}.

Different from the traditional optimization method such as (\ref{2}), this paper will present a novel idea by transforming the optimization into an optimal control theory. The detailed formulation is as follows: 

We consider the discrete-time linear time-invariant system 
\begin{equation}
	{x_{k + 1}} = {x_k} + {u_k}, \label{3}
\end{equation}
where $x_k$ is the $n$-dimensional state, $u_k$ is the $n$-dimensional control, which can indeed be perceived as an iterative update gain, which is to be further specified later. We transform the task of finding solutions to problem (\ref{1}) into the updating of the state sequence {${x_k}$} within the optimal control problem, i.e.,
\begin{equation}
		\begin{array}{l}
		{\rm minimize} {\rm{ }}\sum\limits_{k = 1}^N {(f({x_k}) + \frac{1}{2}u_k^\mathrm {T}} Ru_k^{}) + f({x_{N + 1}}),\\
		{\rm subject\ to}{\rm {\; (3)}},
\end{array}\label{4}
\end{equation}
where the initial condition ${x_0}$ is given, $N$ is the time horizon. The terminal cost is $f({x_{N+1}})$ and the control weighted matrix $R\in\mathbb{S}^n_{++}$. The goal of the optimal control problem is to find an admissible control sequence \{${u_k}$\} which minimizes the long-term cost.

	As mentioned earlier, we consider the solution \{${u_k}$\} of problem (\ref{4}) as the variation in sequence \{${x_k}$\} from ${x_0}$ to ${x^*}$, whereas our objective is to attain the steady state ${x^*}$.

\begin{remark}
It's readily apparent from (\ref{4}) that we reduced the accumulation of ${f(x_k)}$ and ${u_k^ \mathrm {T} R{u_k}}$.
 This signifies that we will strike a balance between minimizing control energy consumption and reaching the minimum value of $f(x_k)$.
 Considering the update formula (\ref{3}) for ${x_k}$, the control sequence \{${u_k}$\} must guide ${x_k}$ toward the local minimum point of ${f(x_k)}$ with small control energy consumption.  This effectively establishes a connection with the optimization problem. A more detailed discussion will be conducted in Section \ref{123}.
\end{remark}
\section{OPTIMIZATION MOTHOD USING OPTIMAL CONTROL }
In this section, we will solve the optimal control problem (\ref{4}) by applying Pontryagin's Maximum Principle\cite{pontryagin2018mathematical}. 
The resulting optimal steady state of system \eqref{3} can recover one of the local minimum point of optimization problem (\ref{1}).
All minimum points can always be obtained by adjusting the input weight matrix $R$ of the optimal control problem \eqref{4}.

\subsection{Analytical Solution} 

Because the optimal control problem (\ref{4}) essentially focuses on finding $u_k$ to minimize $f(x_k)$ and use as energy $u_k^{T} Ru_k$ as possible. The optimal state of problem (\ref{4}) can be used to describe a local minimum point of problem (\ref{1}). This establishes a connection between the optimization problem and the optimal control problem.
Then, we will apply the optimal control theory to solve the problem (\ref{4}), leading to the following theorem.
\begin{theorem}	\label{th1}
	The local minimum point of problem (\ref{1}) can be characterized by the following update relation:
	\begin{equation}
		{x_{k + 1}^*} = {x_k^*} + {u_k^*},	x_0^* = x_0,\label{6}
	\end{equation}where 
	\begin{equation}
		u_k^* =  - {{R^{ - 1}}}\sum\limits_{i = k + 1}^{N + 1}  \nabla f(x_i^*).\label{5}
	\end{equation}
\end{theorem}
\begin{proof}
	Based on the aforementioned relationship, to solve problem (\ref{4}), define the \textit{Hamiltonian} :
	\begin{equation}
		H({x_k},{u_k},{\lambda _{k + 1}}) = f({x_k}) + \frac{1}{2}u_k^ \mathrm {T} R{u_k} + 
		\lambda _{k + 1}^ \mathrm {T} ({x_k} + {u_k}),\label{7}
	\end{equation}where ${\lambda _k}$ is the \textit{n}-dimensional costate.
	Indeed, the costate $\lambda_{k}$ assumes the function of Lagrange multipliers \cite{li2017maximum}.
	
	\noindent By applying the Pontryagin's Maximum Principle, we can derive the following FBDEs:
	
\begin{equation}
				{x_{k + 1}^*} = {x_k^*} + {u_k^*},\label{8}
	\end{equation}
\begin{equation}
		\lambda _k^* = \nabla f(x_k^*) + \lambda _{k + 1}^* ,\label{9} 
	\end{equation}
\begin{equation}
		x_0^* = x_0, \lambda_{N+1}^* = \nabla f(x_{N+1}^*), \label{10}
	\end{equation}along with the equilibrium condition
\begin{equation}
	Ru_k^* + \lambda _{k + 1}^* = 0.\label{12}
\end{equation}
Let $k\leftarrow k+1$, utilizing the iterative equation (\ref{9}) and terminal condition (\ref{10}), we have
\begin{equation}
	\lambda _{k+1}^* = \sum\limits_{i = k + 1}^{N + 1}  \nabla f(x_i^*).\label{13} 
\end{equation}

\noindent By substituting (\ref{13}) into (\ref{12}), the optimal controller admits:
	\begin{equation}
		u_k^* =  - {{R^{ - 1}}}\sum\limits_{i = k + 1}^{N + 1}  \nabla f(x_i^*).\label{14}
	  \end{equation}The proof is now completed.
	\end{proof}
\begin{remark}
Because of noncausality, it is not used to obtain the optimal state directly.
\end{remark}
\begin{remark}
Each local minimum point can be associated with the optimal control problem (4) of different input weight matrix $R$. In contrast, gradient descent method finds various minimum points by adjusting the step size blindly.  
\end{remark}

\subsection{Numerical Solution} \label{555}
It's hard to calculate (\ref{6})-(\ref{5}) analytically. However, the numerical calculation can be achieved by solving the FBDEs. Enlightened by \cite{li2017maximum}, we thus provide a numerical solution algorithm, which is summarized as follows:

\begin{algorithm}
	\renewcommand{\algorithmicrequire}{\textbf{Input:}}
	\renewcommand{\algorithmicensure}{\textbf{Output:}}
	\caption{The numerical algorithm for solving problem (\ref{4})}
	\label{alg1}
	\begin{algorithmic}[1]
		\STATE Initialization: \{$u_k^0$\}, $k=0,1,...,N$ , $x_0$, $\alpha$, $t \leftarrow 0$, $\varepsilon$ 
		\REPEAT
		\STATE Forward Update \{$ x_{k}^{t} $\} based on Equation~(\ref{17})
		\STATE Backward Update \{$\lambda _{k}^{t}$\} based on Equation~(\ref{18})
		\STATE Calculating $\frac{{\partial H(x_k^t,u_k^t,\lambda _{k + 1}^t)}}{{\partial u_k^t}}$ from \{$ x_{k}^{t} $\} and \{$\lambda _{k}^{t}$\}.
		\STATE Update 
		\begin{equation}
			u_k^{t + 1} = u_k^t - \alpha \frac{{\partial H(x_k^t,u_k^t,\lambda _{k + 1}^t)}}{{\partial u_k^t}} \label{19}
		\end{equation}
		\STATE $t \leftarrow t + 1$   
		\UNTIL $||\frac{{\partial H(x_{k}^{t},u_{k}^{t},\lambda _{k+1}^{t})}}{{\partial u_{k}^{t}}}|| < \varepsilon $
		\ENSURE  \{$ x_{k}^t $\}, \{$u_k^t$\}
	\end{algorithmic}  
\end{algorithm}
During the initialization phase, a set of control sequences \{$u_k^0$\}, step size $\alpha$, and error $\varepsilon$ are given, and the initial state $x_0$ is known. Using the forward equation
\begin{equation}
	{x_{k + 1}^t} = {x_k^t} + {u_k^t}, x_0^t=x_0.\label{17}
\end{equation}$\{x_k^t\}$ can be acquired.

Subsequently, \{$\lambda _{k}^{t}$\} can be computed based on \{$ x_{k}^{t} $\} and the backward equation 
\begin{equation}
	\lambda _k^t = \nabla f(x_k^t) + \lambda _{k + 1}^t, \lambda_{N+1}^t = \nabla f(x_{N+1}^t).\label{18} 
\end{equation}

According to \{$ x_{k}^{t} $\}, \{$\lambda _{k+1}^{t}$\} and $\frac{{\partial H(x_k^t,u_k^t,\lambda _{k + 1}^t)}}{{\partial u_k^t}}$, the new sequences \{$u_k^{t+1}$\} can be obtained from (\ref{19}). This iterative process continues until the algorithm converges. Upon the completion of this algorithm, we can use the control sequences along with (\ref{3}) to compute \{$ x_{k}$\}, and the steady state $x^*$ can then be regarded as the solution to the optimization problem (\ref{1}).

\subsection{Dissussion of The Proposed Optimization Method} \label{123}
In this subsection, we discuss the compelling characteristics of solving optimization problems using the optimal control theory as follows:
\begin{itemize}
	
	\item 
	The selection of our input weight matrix $R$ will not result in divergence of \{$x_k$\}. When $R$ is smaller, $x_k$ can converge to the global or local minimum points with few update iterations. 
From (\ref{19}), our method is designed in a way that prevents the convergence of $x_k$ towards the local maximum points or saddle points of the function $f(x)$.
	\item Our method alleviates oscillations in the iteration process of \{$x_k$\}. Such oscillations, which occur near the local minimum point, would contradict the fundamental objective of minimizing the cost function.
	\item For some non-convex functions, given a judicious choice of initial value $x_0$, we can make \{$x_k$\} converge towards different local minimum points by adjusting the matrix $R$. A larger value of $R$ can cause \{$x_k$\} to converge to a local minimum point closer to $x_0$, while a smaller value of $R$ can enable \{$x_k$\} to converge to a local minimum point farther away from $x_0$. 

\end{itemize}

These characteristics are actually guaranteed by the cost function of the optimal control problem (\ref{4}).
It will be further demonstrated in the experimental results of Section \ref{222}.

\section{NUMERICAL EXPERIMENTAL} \label{222}
In this section, we present preliminary computational results for the numerical
performance analysis of our proposed method and demonstrate (a) the better convergence of our proposed method compared with gradient descent and Newton's method in both convex and non-convex functions, (b) the high accuracy of our method, (c) escaping saddle points or local maxima and (d) applicable to nonconvex functions and multivariable situation.
\subsection{Fast Convergence}
When $R$ is smaller, $x_k$ can converge the global or local minimum points with fewer iterations by using our method. Choosing the non-convex function 
\[f_1(x) = {x^4} + \sin {x}\] with the global unique minimum point at $x^*=-0.592$ and an initial value of $x_0 = 10$. We set $R = 1$ and $R=200$ for the optimal control method.
It can be seen from Fig. \ref{fig_3} that when $R=1$ the algorithm converges to the minimum point in nearly 10 iterations, while for $R=200$, it takes approximately 70 iterations to reach the same minimum point.
  Given the arbitrariness in the choice of $R$ in our method, it's advisable in general to opt for smaller values of $R$ to minimize the number of iterations.
\begin{figure}[htbp]
		\vspace{-0.4cm}
	\centering
	\includegraphics[width=0.4\textwidth,height=0.3\textwidth]{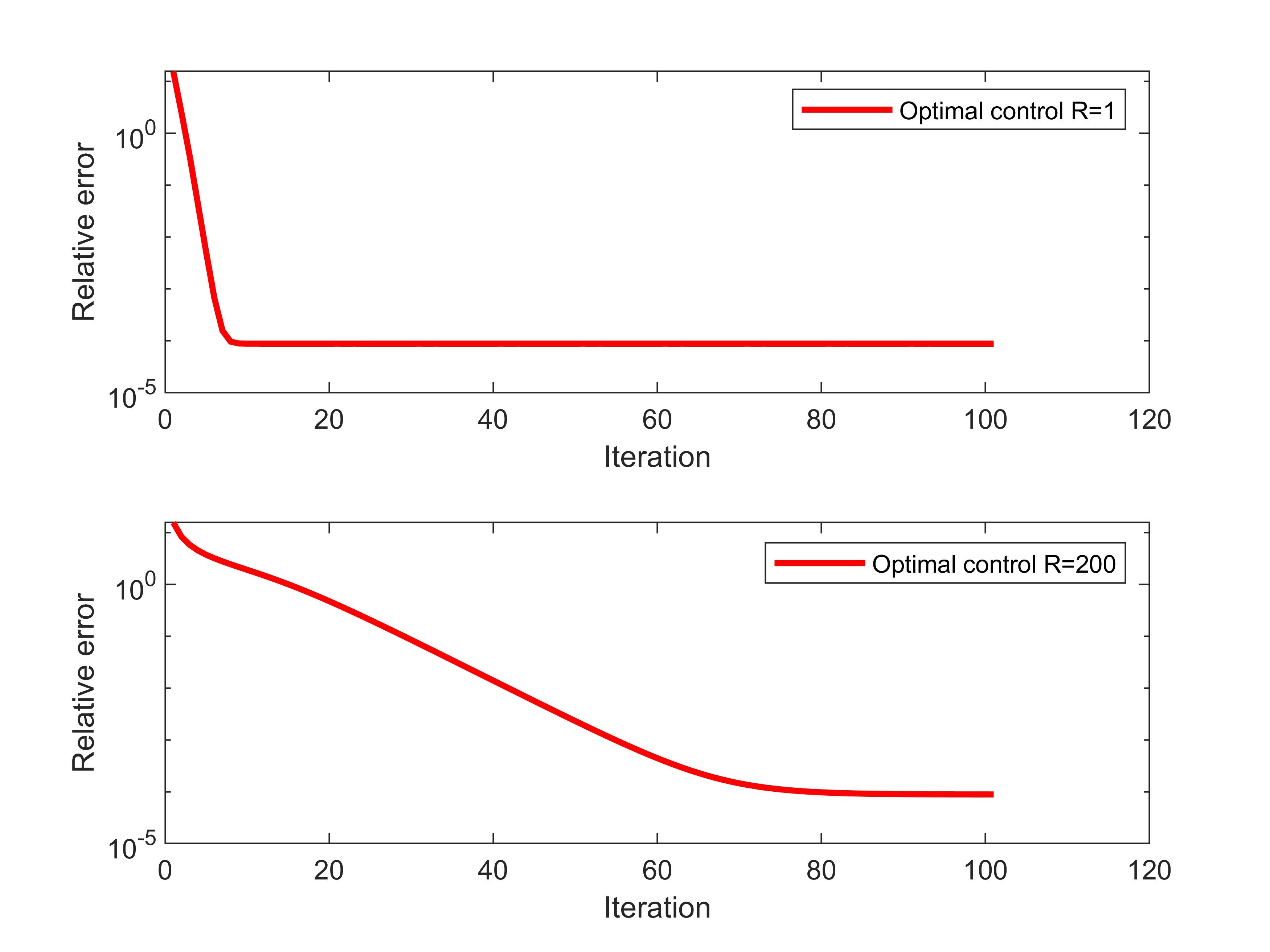}
	\caption{Iteration trajectory of \{$x_k$\} for $f_1(x)$ with $R=1$, $200$}
	\label{fig_3}
\end{figure}

 The basic formula for Newton's method is as follows:
\begin{equation}
	{x_{k + 1}} = {x_k} - {({\nabla ^2}f(x))^{ - 1}}\nabla f(x).
\end{equation}

Let's consider the case of a convex function. We choose the function 
\[{f_2}(x) = {e^x} + \sin x + {x^2}\] with an initial value of $x_0 = 3$ and global minimum point $x^*=-0.6558$.
 The initial step size for gradient descent is set to $\eta = 0.1$, and we set $R = 0.01$ for our method. As shown in Fig. \ref{fig_1}, our method and Newton's method converge in nearly 5 iterations. It's important to note that if the gradient descent step size is chosen too large, it can lead to oscillations during the iterative process.
\begin{figure}[htbp]
		\vspace{-0.4cm}
	\centering
	\includegraphics[width=0.4\textwidth,height=0.2\textwidth]{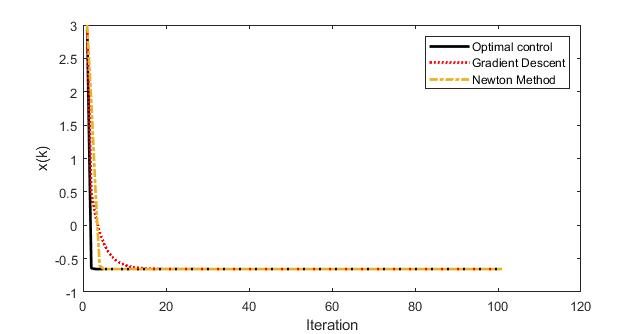}
	\caption{Iteration trajectory of \{$x_k$\} for $f_2(x)$ with $R=0.01$ and $\eta=0.1$}
	\label{fig_1}
\end{figure}

Taking a non-convex function \[f_3(x) = \ln ({x^2} + 1) + \ln ({(x - 1)^2} + 0.01)\] with the global unique minimum point at $x^*= 0.995$ and an initial value of $x_0 = 2$. Fig. \ref{fig_2} depicts the iterative trajectory of the gradient descent and our proposed method. The initial step size for gradient descent is set to $\eta = 0.01$ and $R = 0.01$ for our method. It's evident from the figures that the gradient descent experiences oscillations, whereas the optimal control algorithm achieves convergence in just about 5 iterations. Our method maintains a more favorable convergence behavior. Due to the convexity of $f_3(x)$, Newton's method diverges. We will not present the graphical results of Newton's method.
\begin{figure}[htbp]
		\vspace{-0.4cm}
	\centering
		\includegraphics[width=0.4\textwidth,height=0.2\textwidth]{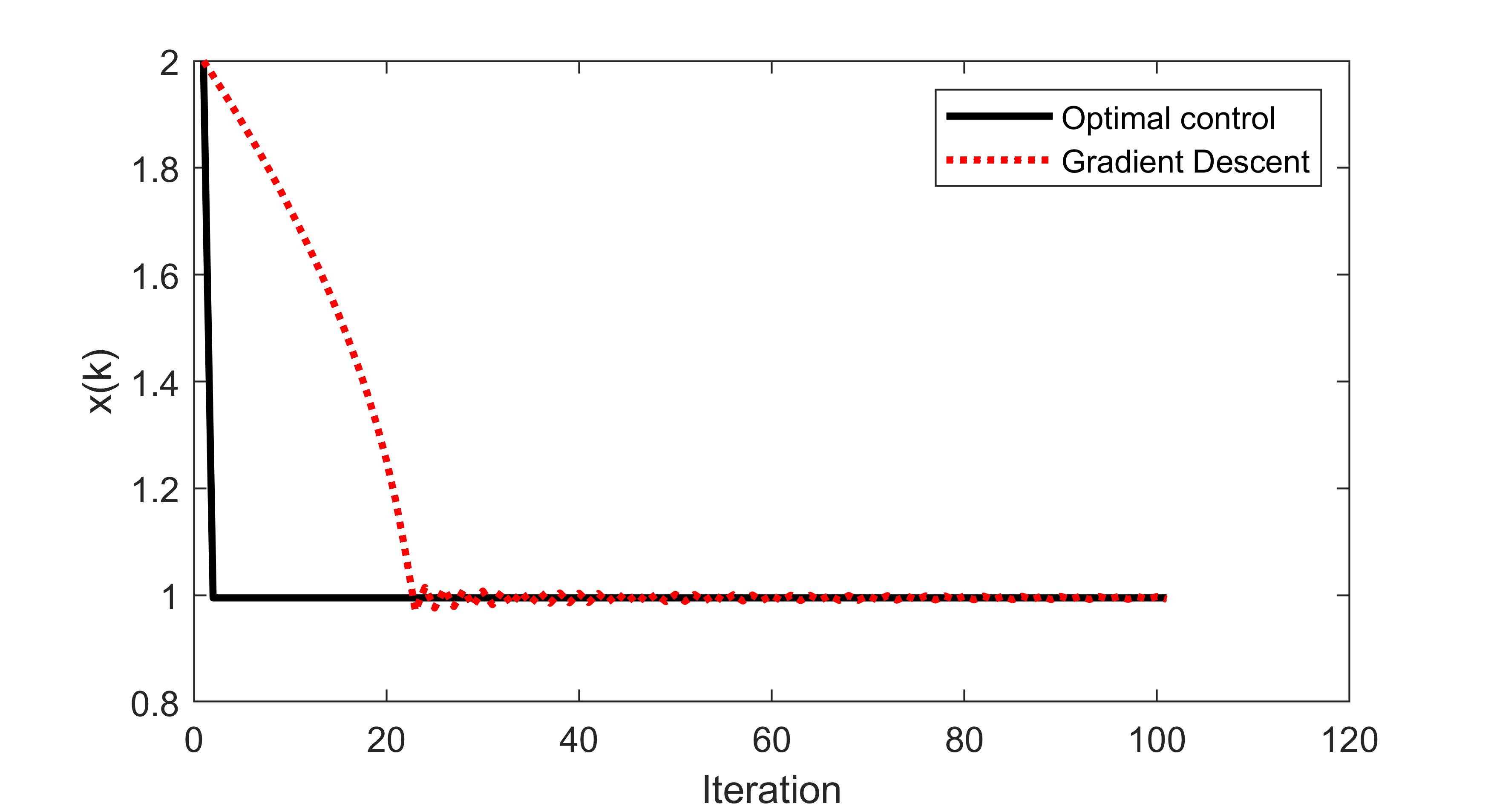}
	\caption{Iteration trajectory of \{$x_k$\} for $f_3(x)$ with $R=0.01$ and $\eta=0.01$}
	\label{fig_2}
\end{figure}

%

\subsection{High Accuracy}
This subsection will discuss the higher convergence accuracy of our proposed method compared to gradient descent and Newton's method.

Fig. \ref{fig_4} (a) shows the relative error of gradient descent, Newton's method, and our method for $f_1(x)$ with $\eta=0.001$ and $R=0.01$. Fig. \ref{fig_4} (b) illustrates the relative error of the gradient descent, Newton's method, and our method for $f_2(x)$ with $R=0.01$ and $\eta=0.1$. Fig. \ref{fig_4} (c) respectively shows the relative error of gradient descent and our method for $f_3(x)$ with $R=0.01$ and $\eta=0.01$, Newton's method diverges. It can be observed that our algorithm demonstrates higher precision.
\begin{figure}[htbp]
	\centering
		\subfloat[Relative error of \{$x_k$\} for $f_1(x)$ with $R=0.01$ and $\eta=0.001$]{
				\includegraphics[width=0.4\textwidth,height=0.25\textwidth]{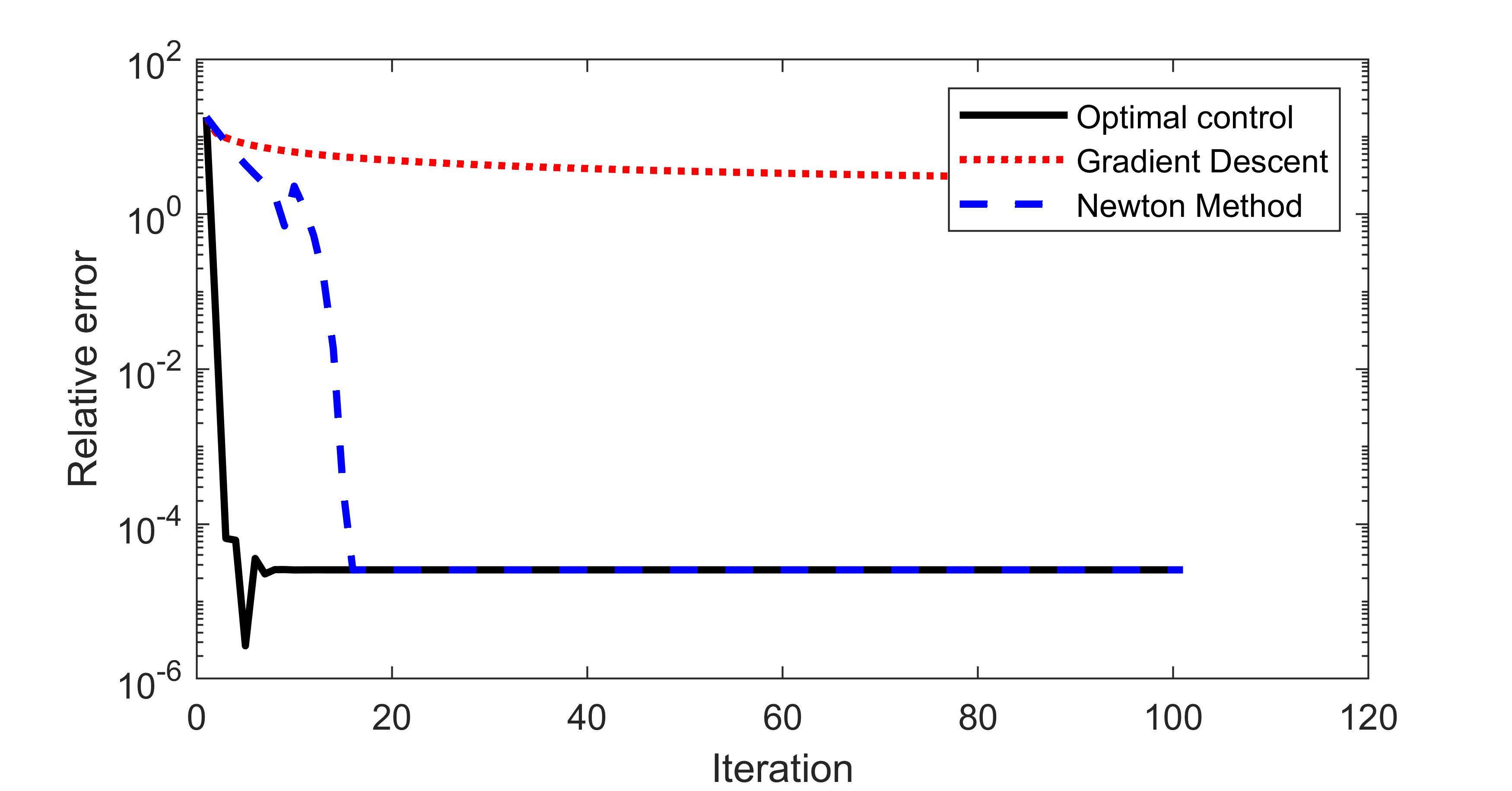}}\\
	\subfloat[Relative error of \{$x_k$\} for $f_2(x)$ with $R=0.01$ and $\eta=0.1$]{
			\includegraphics[width=0.4\textwidth,height=0.25\textwidth]{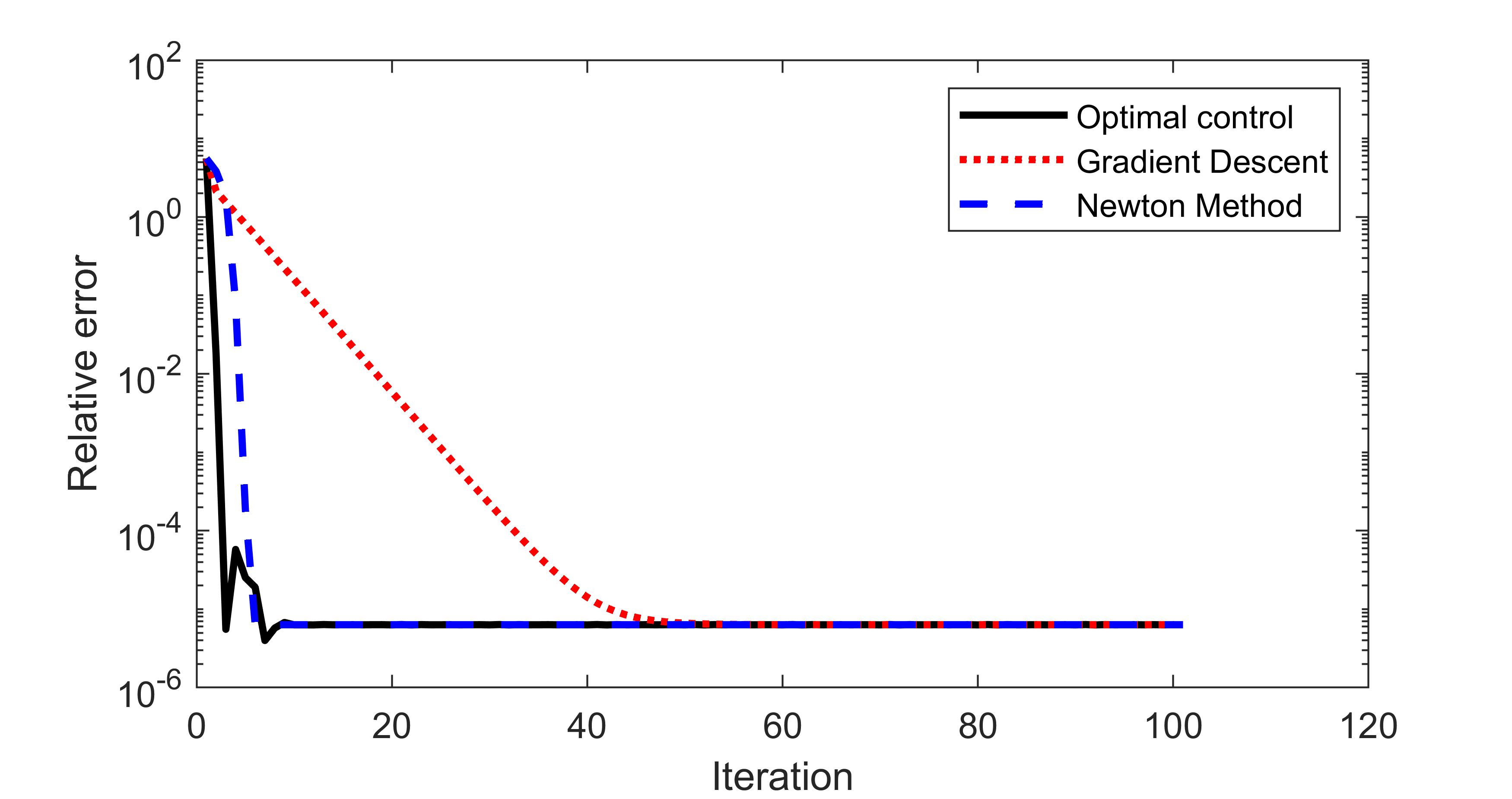}}\\
			\subfloat[Relative error of \{$x_k$\} for $f_3(x)$ with $R=0.01$ and $\eta=0.01$]{
			\includegraphics[width=0.4\textwidth,height=0.25\textwidth]{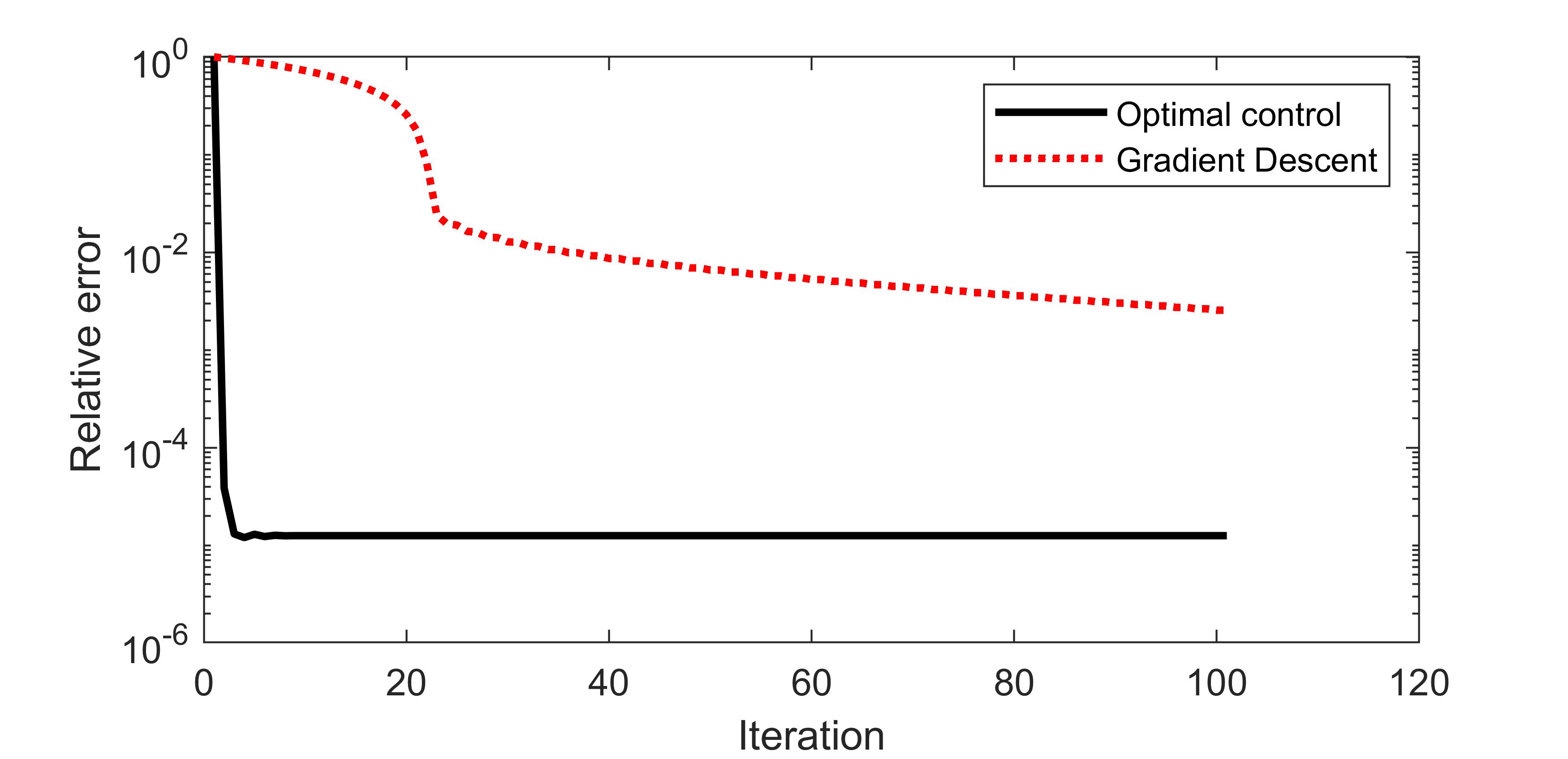}}
	\caption{Relative error $(({x_k} - {x^*})/{x^*}$)}
	\label{fig_4}
	\vspace{-0.6cm}
\end{figure}
\subsection{Escaping The Saddle Point}
We will show whether \{$x_k$\} from our proposed method can converge to the optimal solution when a saddle point is chosen as the initial value. Consider the  function 
\[f_4(x) = 7{x^3} + {x^4} + {{\rm{e}}^{{x^2}}} + {e^{ - {x^2}}}\]and set the initial value $x_0 = 0$, which is the saddle point of $f_4(x)$. This function has a global unique minimum point $x^*=-1.566$. The initial step size for gradient descent is set to $\eta = 0.026$ and we set $R = 0.026^{-1}$ correspondingly. It can be observed that the optimal control does not remain at the saddle point but converges towards the minimum point, whereas gradient descent and Newton's method remain at the saddle point in Fig. \ref{fig_6}.
\begin{figure}[H]
	\centering
	\includegraphics[width=0.4\textwidth,height=0.24\textwidth]{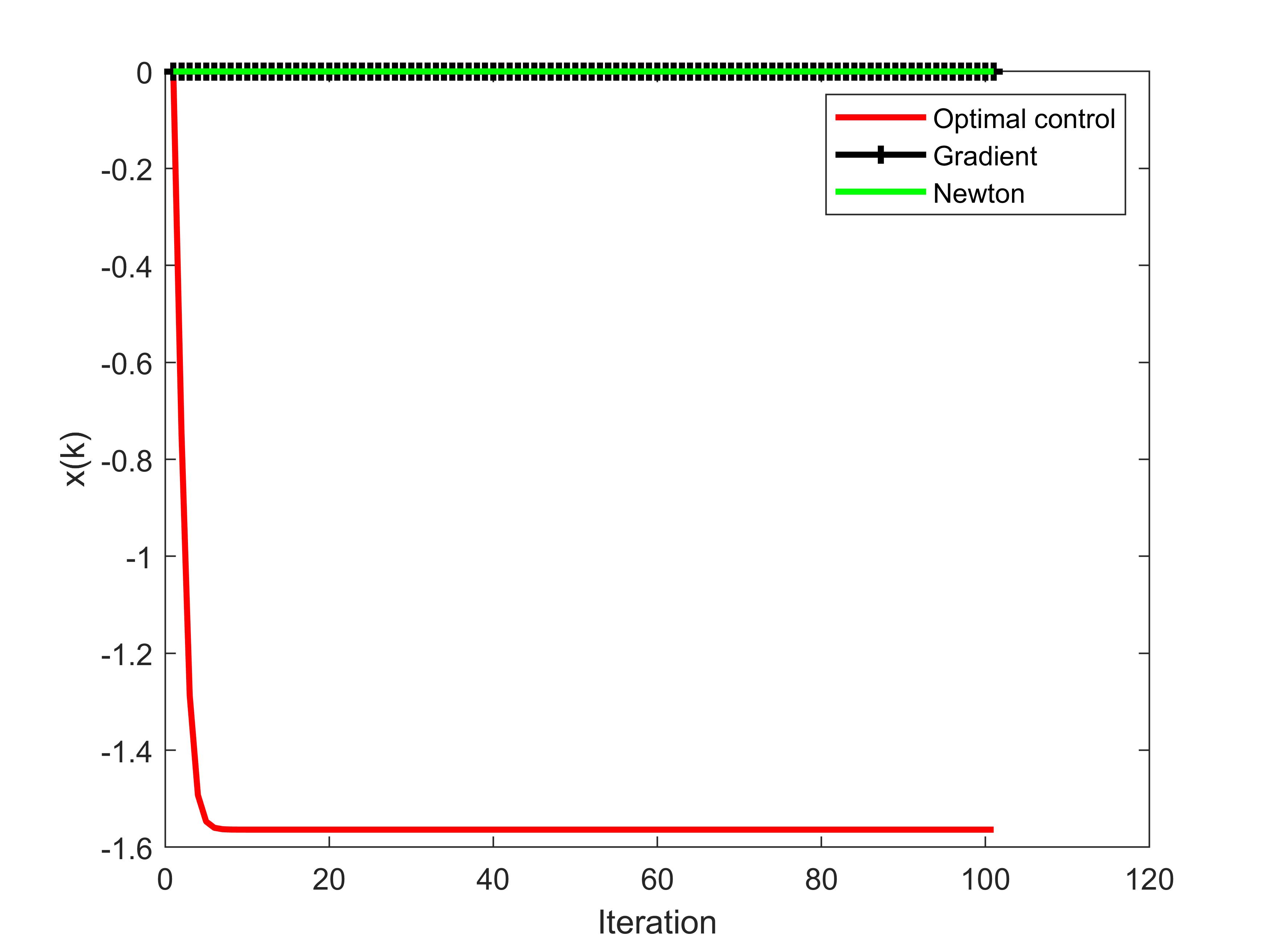}
	\caption{Iteration trajectory of \{$x_k$\} for $f_4(x)$ with $R=0.026^{-1}$ and $\eta=0.026$}
	\label{fig_6}
\end{figure}
\subsection{Applicable to Nonconvex Functions} \label{111}
  A larger value of $R$ can cause $x_k$ to converge to a local minimum point closer to $x_0$, while a smaller value of $R$ can lead $x_k$ to converge to a local minimum point farther away from $x_0$. 
 It becomes apparent that the $R$ is the weight of the problem (\ref{4}) as discussed in \ref{123}.
    We choose the function 
\[f_5(x) = x - 4{x^2} + 0.2{x^3} + 2{x^4}\] to illustrate this phenomenon. This function has a local minimum point $x^*_1=0.89$ and a global minimum point $x^*_2=-1.094$. We set $x_0 = -10$, $R = 100$ and $R = 0.1$. It can be observed that our method leads $x_k$ to converge to different local minimum points in Fig. \ref{fig_61}.
\begin{figure}[htbp]
	\vspace{-0.2cm}
	\centering
		\includegraphics[width=0.5\textwidth,height=0.84\textwidth]{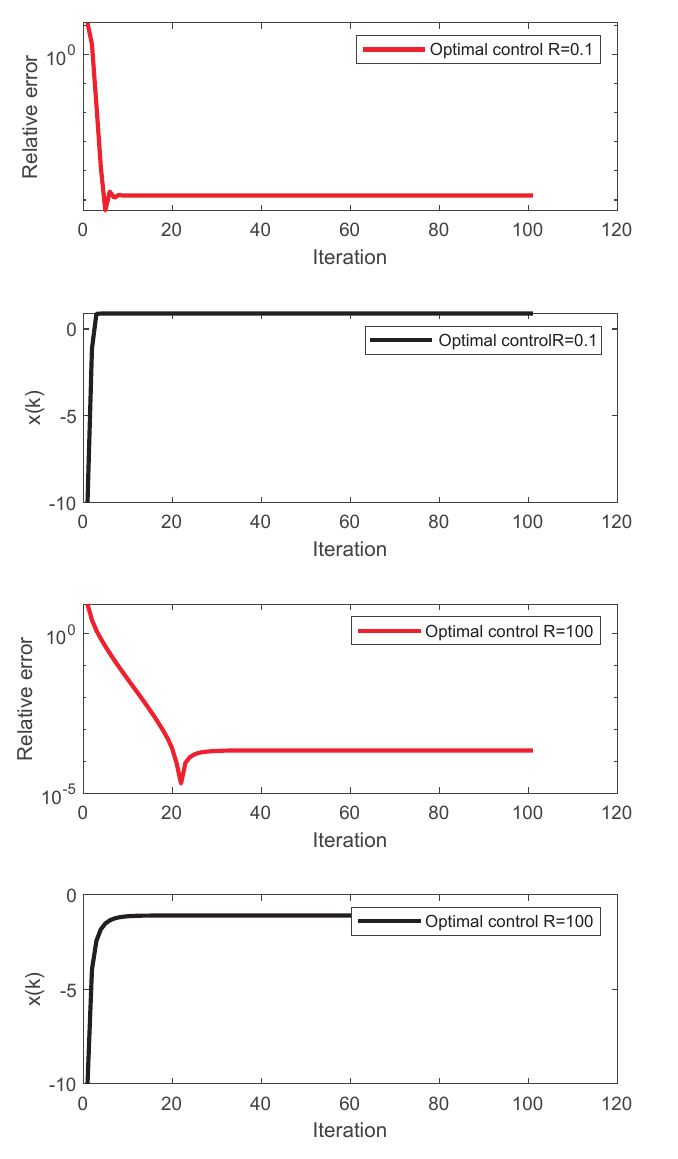}
		\caption{Iteration trajectory and relative error of \{$x_k$\} for $f_5(x)$ with $R=100,0.1$}
	\label{fig_61}
	\end{figure}

Consider the non-convex function
\[f_6(x) = (x - 1)(x + 1)(x + 0.5)(x + 1.5)(x - 0.5)(x - 1.5)\] that possesses three local minimum points $x^*_1=1.323$, $x^*_2=0$, $x^*_3=-1.323$. Take an initial value of $x_0=-3$. In the case of $R = 1,200,500$, \{$x_k$\} obtained from the proposed method convergence to $x^*_1$, $x^*_2$, and $x^*_3$ respectively. The corresponding iteration and relative error are shown in Fig. \ref{fig_7}.

\begin{figure}[H]
\vspace{-0.1cm}
	\centering
		\includegraphics[width=0.5\textwidth,height=0.943\textwidth]{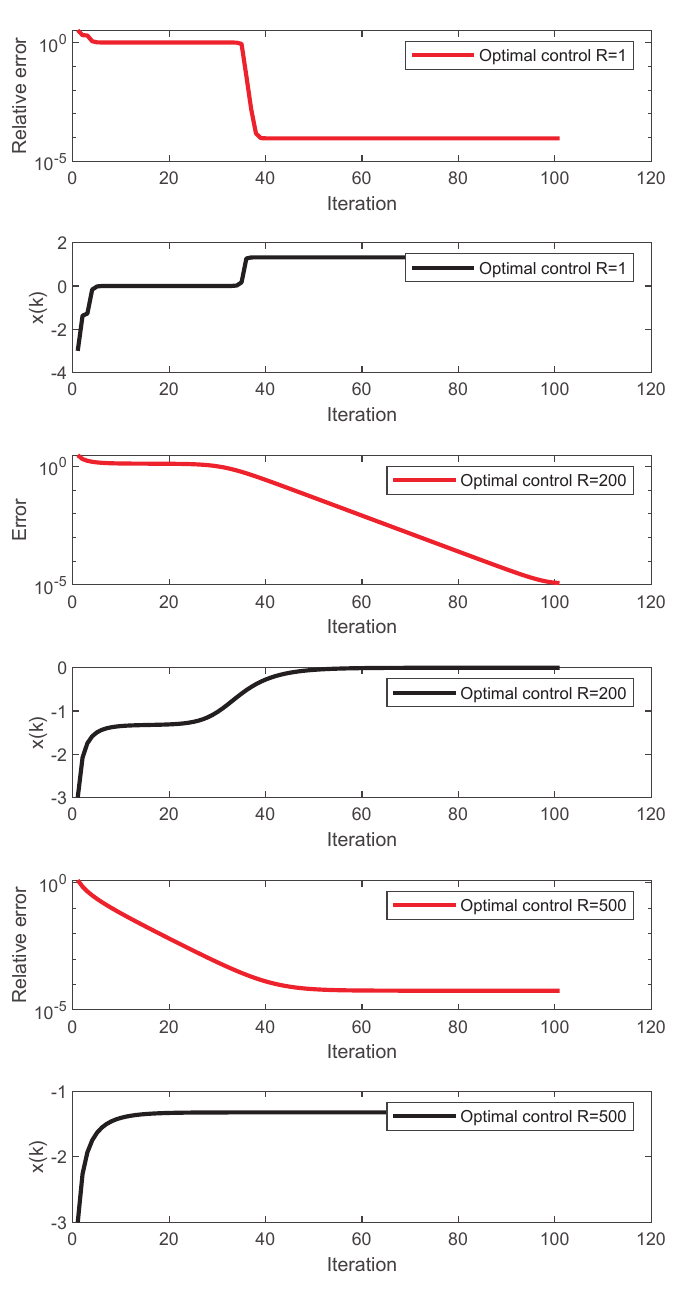}
	\caption{Iteration trajectory and relative error of \{$x_k$\} for $f_6(x)$ with $R=1,200,500$}
	\label{fig_7}
\end{figure}

\begin{remark}
	It can be observed that the $x_k$ moves to the local minimum point far away from $x_0$ when we use a smaller $R$, it may stay at other extreme points for a while and then leave, as shown in Fig. \ref{fig_7}. This is a very interesting thing and worthy of our subsequent research. In particular, it should also be pointed out that at present, we only know how to adjust $R$ from small to large, but the specific thresholdthat makes $x_k$ converge to different local minimum points remains to be studied and proved.
\end{remark}

\subsection{Applicable to Multivariable Function}

Finally, we will illustrate that our method is still valid for multivariable function by using a non-convex function 
\[{f_7}(x,y) = {x^4} + {y^4} + \sin x .\]
 We initialize with $[x_0;y_0]=[2;-2]$ and set $R=0.12^{-1}{I_2}$, $\eta=0.12$. This function has the global unique minimum point $[-0.592;0]$. The results depicted in Fig. \ref{fig_13} demonstrate that our method exhibits almost no oscillations compared to gradient descent.
\begin{figure}[htbp]
		\centering
		\includegraphics[width=0.5\textwidth,height=0.3\textwidth]{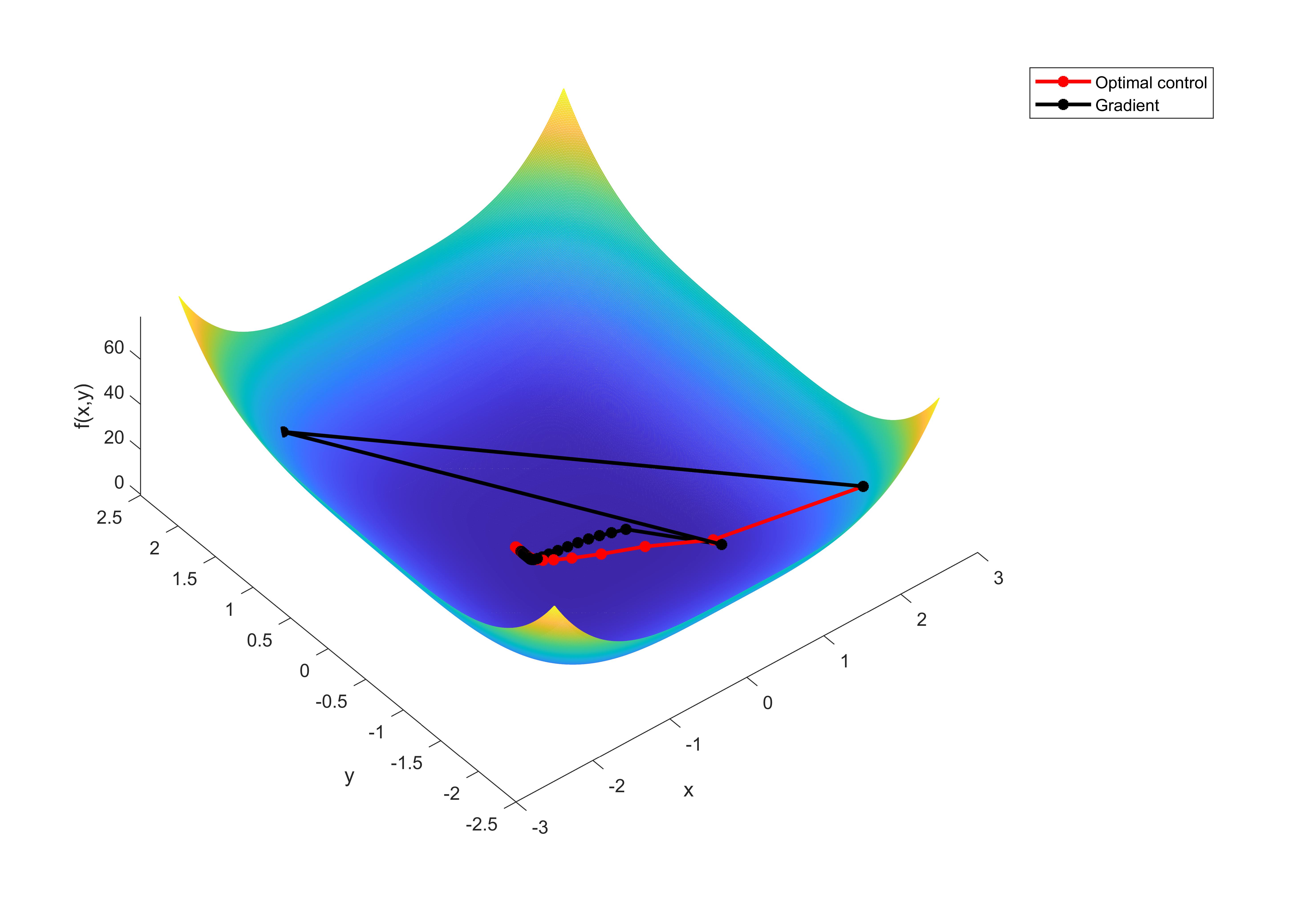}
		\caption{Iteration trajectory of \{$x_k$\} for $f_7(x)$ with 
			$R={0.12^{ -1}}{I_2}$ and $\eta=0.12$}
		\label{fig_13}
	\end{figure}

In order to make \{$x_k$,$y_k$\} converge towards different local minimum points by adjusting the matrix $R$, We adopt an alternating optimization approach.
Let's consider the function
\[{f_8}(x,y) = \ln ({x^2} + {y^2} + 1) + \ln ({(x - 10)^2} + {(y - 10)^2} + 1)\]
\qquad\qquad ${\rm{     }}+ \ln ({(x - 2)^2} + {(y - 30)^2} + 1)$

\noindent with $[x_0;y_0]=[-20;40]$. This function has three local mimimum points $[x^*_1;y^*_1]=[2;29.9]$, $[x^*_2;y^*_2]=[0.05;0.08]$ and $[x^*_3;y^*_3]=[9.93;9.99]$.
Different from the single variable function optimization, the carried out by decomposing it along the two directions of $x$ axis and $y$ axis.
The iteration will be decomposed into a series of steps, where each variable is optimized separately. Initially, the algorithm optimizes in the $x$ axis direction or the $y$ axis direction, solving the FBDEs separately in each direction. Finally solving FBDEs in two directions to achieve convergence towards the direction of different local minimum points. Refer to Fig. \ref{fig_14} for detailed visualizations.

\begin{figure}[htbp]
	\centering
	\subfloat[{Iteration trajectory converges to $[x^*_3;y^*_3]$}]{
		\includegraphics[width=0.5\textwidth,height=0.3\textwidth]{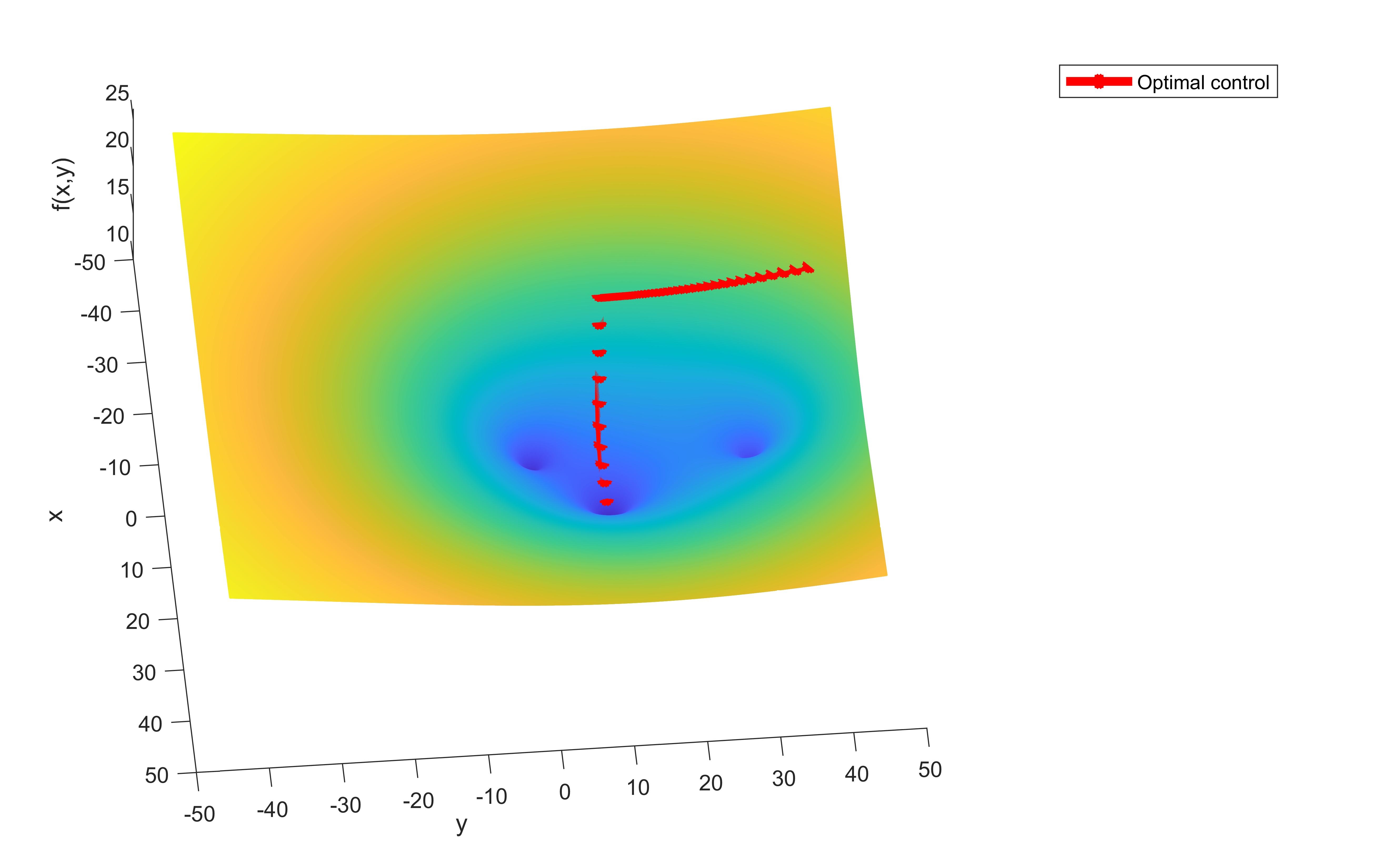}}\\
	\subfloat[{Iteration trajectory converges to $[x^*_2;y^*_2]$}]{
		\includegraphics[width=0.5\textwidth,height=0.3\textwidth]{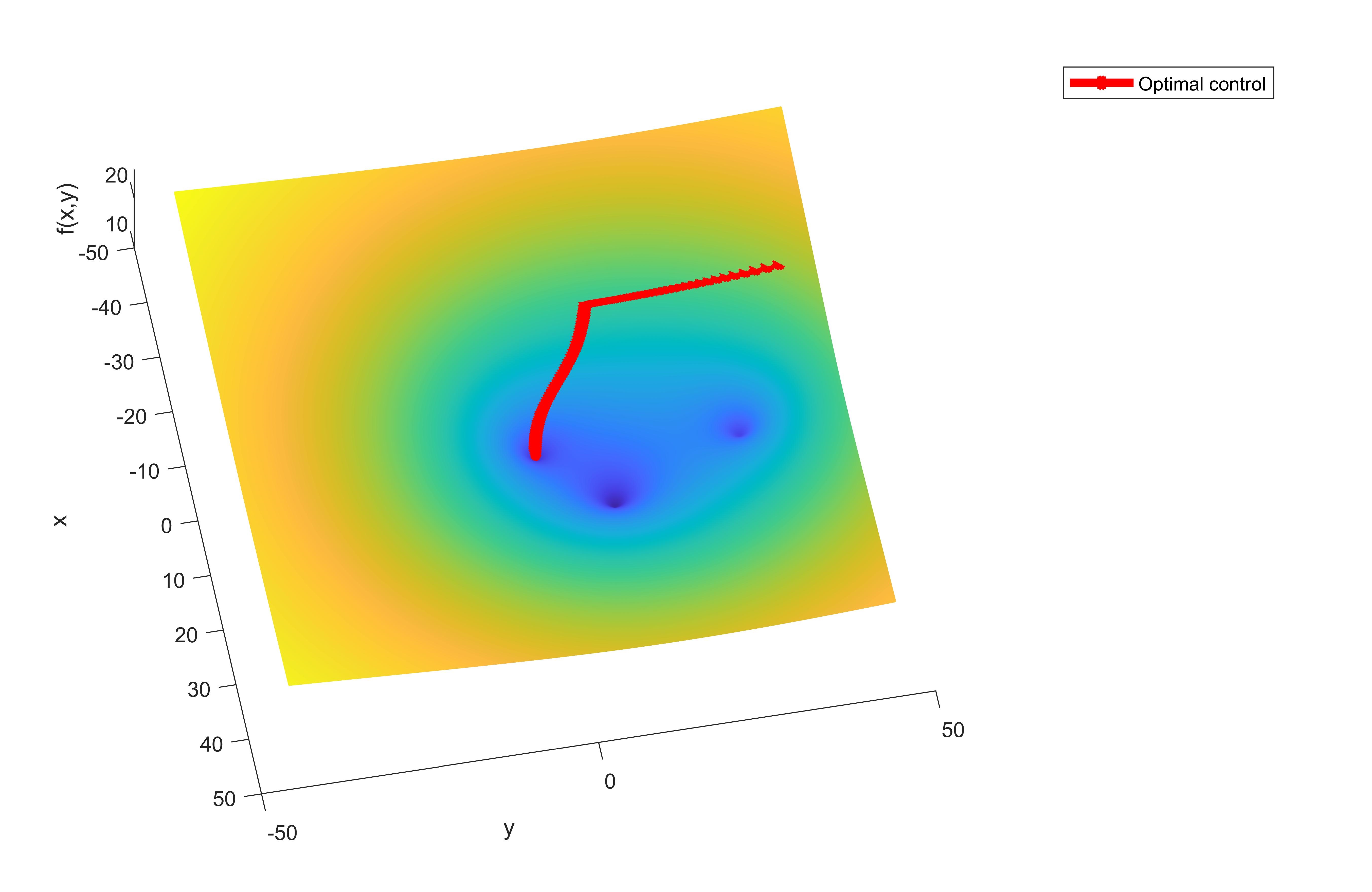}}\\
	\subfloat[{Iteration trajectory converges to $[x^*_1;y^*_1]$}]{
		\includegraphics[width=0.5\textwidth,height=0.3\textwidth]{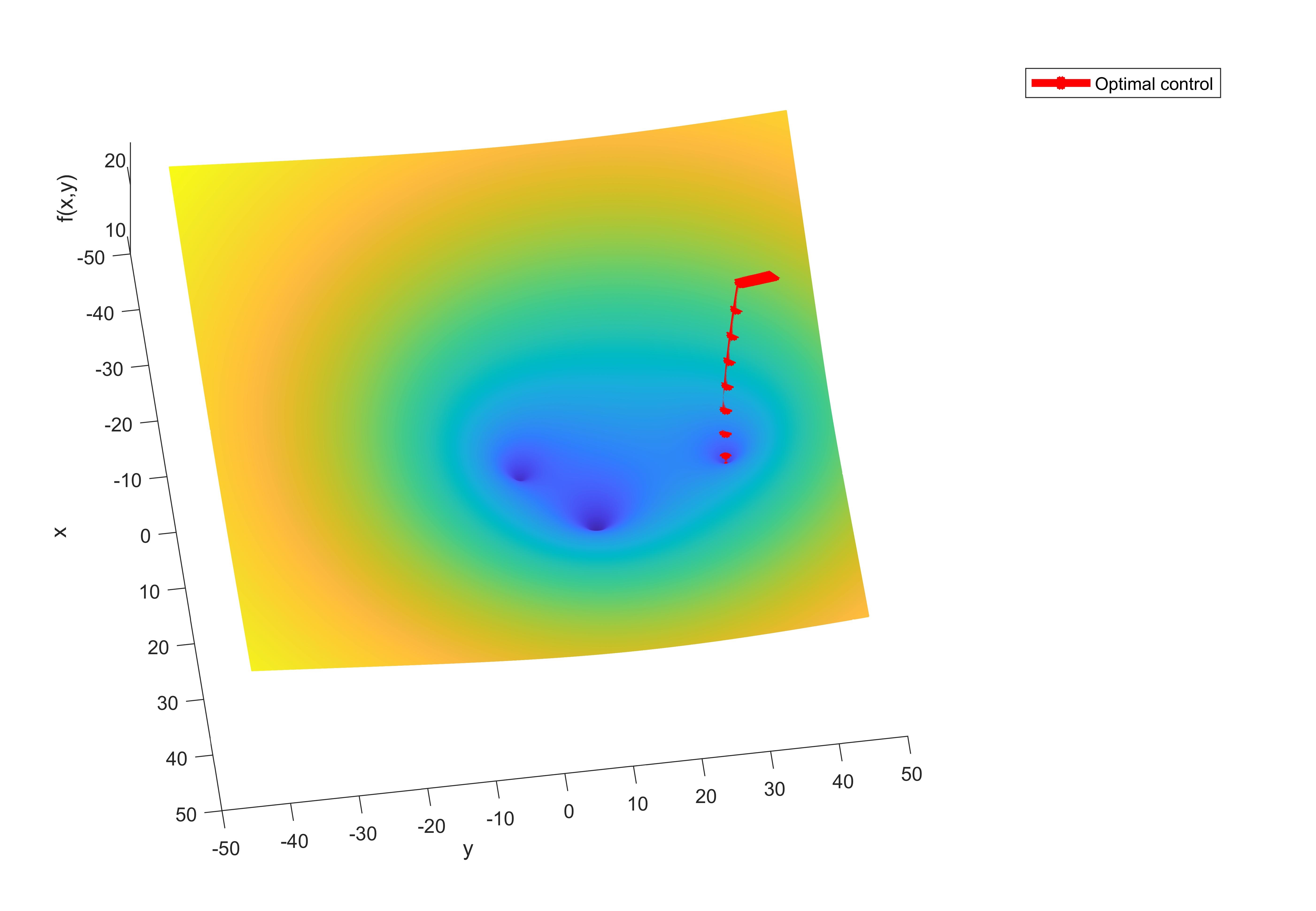}}\\
	\caption{Iteration trajectory of \{$x_k$\} for $f_8(x)$}
	\label{fig_14}
\end{figure}

In Fig. \ref{fig_14} (a), we first fix the variable $x$ and move it along the $y$ axis direction by setting $R = 1$, it reaches the point $[-20;11]$. Then we make its move along the $x, y$ axis direction by setting $R = I_2$ to reach the local minimum point $[9.93;9.99]$. Following an approach similar to Fig. \ref{fig_14} (a), we initially move it along the $y$ direction with $R = 1$, ultimately reaching the point $[-20;11]$. Then we make its move along the $x, y$ axis direction by setting $R =$ [100 0;0 0.00001] to reach the local minimum point $[0.05;0.08]$ as shown in Fig. \ref{fig_14} (b). In Fig. \ref{fig_14} (c), we first make its move along the $y$ direction and set $R = 100$, it reaches the point $[-20;35]$. Then we make its move along the $x, y$ axis direction by setting $R = I_2$ to reach the local minimum point $[2;29.9]$.



\section{CONCLUSIONS}
In this paper, we have proposed the method based on optimal control as a novel approach to tackle the optimization problem by designing an appropriate cost function. Our method has demonstrated promising convergence performance and versatility, enabling us to apply this principle to solve various optimization problems.
In the future, we plan to extend our method to systems with additive noise or time-varying input weight matrices $R$. We will also expand our analysis to address challenges in distributed optimization and explore policy optimization (PO) methods for Linear Quadratic Regulators (LQR) and other related problems.
 The execution of the optimal control algorithm involves solving the FBDEs, which can be time-consuming. To address this, it is essential to choose appropriate methods to simplify the solving process. We will continue our research into algorithms for solving these equations, aiming to enhance the computational speed of our method.

\bibliographystyle{IEEEtran}
\bibliography{An_bib}

\begin{thebibliography}{10}
\providecommand{\url}[1]{#1}
\csname url@samestyle\endcsname
\providecommand{\newblock}{\relax}
\providecommand{\bibinfo}[2]{#2}
\providecommand{\BIBentrySTDinterwordspacing}{\spaceskip=0pt\relax}
\providecommand{\BIBentryALTinterwordstretchfactor}{4}
\providecommand{\BIBentryALTinterwordspacing}{\spaceskip=\fontdimen2\font plus
\BIBentryALTinterwordstretchfactor\fontdimen3\font minus
  \fontdimen4\font\relax}
\providecommand{\BIBforeignlanguage}[2]{{%
\expandafter\ifx\csname l@#1\endcsname\relax
\typeout{** WARNING: IEEEtran.bst: No hyphenation pattern has been}%
\typeout{** loaded for the language `#1'. Using the pattern for}%
\typeout{** the default language instead.}%
\else
\language=\csname l@#1\endcsname
\fi
#2}}
\providecommand{\BIBdecl}{\relax}
\BIBdecl

\bibitem{luenberger1984linear}
D.~G. Luenberger, Y.~Ye \emph{et~al.}, \emph{Linear and nonlinear
  programming}.\hskip 1em plus 0.5em minus 0.4em\relax Springer, 1984, vol.~2.

\bibitem{boyd2004convex}
S.~P. Boyd and L.~Vandenberghe, \emph{Convex optimization}.\hskip 1em plus
  0.5em minus 0.4em\relax Cambridge university press, 2004.

\bibitem{nocedal1999numerical}
J.~Nocedal and S.~J. Wright, \emph{Numerical optimization}.\hskip 1em plus
  0.5em minus 0.4em\relax Springer, 1999.

\bibitem{ruder2016overview}
S.~Ruder, ``An overview of gradient descent optimization algorithms,''
  \emph{arXiv preprint arXiv:1609.04747}, 2016.

\bibitem{huo2017asynchronous}
Z.~Huo and H.~Huang, ``Asynchronous mini-batch gradient descent with variance
  reduction for non-convex optimization,'' in \emph{Proceedings of the AAAI
  Conference on Artificial Intelligence}, vol.~31, no.~1, 2017.

\bibitem{bottou2010large}
L.~Bottou, ``Large-scale machine learning with stochastic gradient descent,''
  in \emph{Proceedings of COMPSTAT'2010: 19th International Conference on
  Computational StatisticsParis France, August 22-27, 2010 Keynote, Invited and
  Contributed Papers}.\hskip 1em plus 0.5em minus 0.4em\relax Springer, 2010,
  pp. 177--186.

\bibitem{qian1999momentum}
N.~Qian, ``On the momentum term in gradient descent learning algorithms,''
  \emph{Neural networks}, vol.~12, no.~1, pp. 145--151, 1999.

\bibitem{o2015adaptive}
B.~O’donoghue and E.~Candes, ``Adaptive restart for accelerated gradient
  schemes,'' \emph{Foundations of computational mathematics}, vol.~15, pp.
  715--732, 2015.

\bibitem{duchi2011adaptive}
J.~Duchi, E.~Hazan, and Y.~Singer, ``Adaptive subgradient methods for online
  learning and stochastic optimization.'' \emph{Journal of machine learning
  research}, vol.~12, no.~7, 2011.

\bibitem{kingma2014adam}
D.~P. Kingma and J.~Ba, ``Adam: A method for stochastic optimization,''
  \emph{arXiv preprint arXiv:1412.6980}, 2014.

\bibitem{fletcher1977modified}
R.~Fletcher and T.~L. Freeman, ``A modified newton method for minimization,''
  \emph{Journal of Optimization Theory and Applications}, vol.~23, pp.
  357--372, 1977.

\bibitem{sano2019damped}
T.~Sano, T.~Migita, and N.~Takahashi, ``A damped newton algorithm for
  nonnegative matrix factorization based on alpha-divergence,'' in \emph{2019
  6th International Conference on Systems and Informatics (ICSAI)}.\hskip 1em
  plus 0.5em minus 0.4em\relax IEEE, 2019, pp. 463--468.

\bibitem{broyden1967quasi}
C.~G. Broyden, ``Quasi-newton methods and their application to function
  minimisation,'' \emph{Mathematics of Computation}, vol.~21, no.~99, pp.
  368--381, 1967.

\bibitem{gill1972quasi}
P.~E. Gill and W.~Murray, ``Quasi-newton methods for unconstrained
  optimization,'' \emph{IMA Journal of Applied Mathematics}, vol.~9, no.~1, pp.
  91--108, 1972.

\bibitem{broyden1970convergence}
C.~G. Broyden, ``The convergence of a class of double-rank minimization
  algorithms 1. general considerations,'' \emph{IMA Journal of Applied
  Mathematics}, vol.~6, no.~1, pp. 76--90, 1970.

\bibitem{broyden1970convergence2}
------, ``The convergence of a class of double-rank minimization algorithms 2.
  the new algorithm,'' \emph{IMA journal of applied mathematics}, vol.~6,
  no.~3, pp. 222--231, 1970.

\bibitem{liu1989limited}
D.~C. Liu and J.~Nocedal, ``On the limited memory bfgs method for large scale
  optimization,'' \emph{Mathematical programming}, vol.~45, no. 1-3, pp.
  503--528, 1989.

\bibitem{setiono1995use}
R.~Setiono and L.~C.~K. Hui, ``Use of a quasi-newton method in a feedforward
  neural network construction algorithm,'' \emph{IEEE Transactions on Neural
  Networks}, vol.~6, no.~1, pp. 273--277, 1995.

\bibitem{goldfarb2020practical}
D.~Goldfarb, Y.~Ren, and A.~Bahamou, ``Practical quasi-newton methods for
  training deep neural networks,'' \emph{Advances in Neural Information
  Processing Systems}, vol.~33, pp. 2386--2396, 2020.

\bibitem{lin2007trust}
C.-J. Lin, R.~C. Weng, and S.~S. Keerthi, ``Trust region newton methods for
  large-scale logistic regression,'' in \emph{Proceedings of the 24th
  international conference on Machine learning}, 2007, pp. 561--568.

\bibitem{back1996evolutionary}
T.~Back, \emph{Evolutionary algorithms in theory and practice: evolution
  strategies, evolutionary programming, genetic algorithms}.\hskip 1em plus
  0.5em minus 0.4em\relax Oxford university press, 1996.

\bibitem{danilova2022recent}
M.~Danilova, P.~Dvurechensky, A.~Gasnikov, E.~Gorbunov, S.~Guminov,
  D.~Kamzolov, and I.~Shibaev, ``Recent theoretical advances in non-convex
  optimization,'' in \emph{High-Dimensional Optimization and Probability: With
  a View Towards Data Science}.\hskip 1em plus 0.5em minus 0.4em\relax
  Springer, 2022, pp. 79--163.

\bibitem{polyak1963gradient}
B.~T. Polyak, ``Gradient methods for the minimisation of functionals,''
  \emph{USSR Computational Mathematics and Mathematical Physics}, vol.~3,
  no.~4, pp. 864--878, 1963.

\bibitem{pontryagin2018mathematical}
L.~S. Pontryagin, \emph{Mathematical theory of optimal processes}.\hskip 1em
  plus 0.5em minus 0.4em\relax Routledge, 2018.

\bibitem{li2017maximum}
Q.~Li, L.~Chen, C.~Tai \emph{et~al.}, ``Maximum principle based algorithms for
  deep learning,'' \emph{arXiv preprint arXiv:1710.09513}, 2017.

\end{thebibliography}
%

\end{document}